\documentclass{amsart}
\usepackage[utf8]{inputenc}
\usepackage{amsmath}
\usepackage{amsthm}

\usepackage{float}
\usepackage[a4paper, total={6in, 8in}]{geometry}
\usepackage{amsfonts}
\usepackage{graphicx}

\usepackage[colorinlistoftodos]{todonotes}
\usepackage{ dsfont }
\usepackage{MnSymbol}
\newtheorem{theorem}{Theorem}[section]

\newtheorem{lemma}[theorem]{Lemma}

\newtheorem{remark}[theorem]{Remark}

\usepackage[shortlabels]{enumitem}
\usepackage{accents}

\usepackage{subcaption}
\usepackage{comment}
\usepackage{verbatim}
\def\R{\mathbb R}

\def\pa{\partial}
\def\na{\nabla}

\def\div{\mathrm{div}\,}
\def\supp{\mathrm{Supp}\,}

\title{Incompressible Limit for an Age-Structured Tumor Model}
\author{Maeve Wildes}
\date{October 2025}

\begin{document}

\begin{abstract}
In this paper, we consider an age-structured mechanical model for tumor growth. This model takes into account the life-cycle of tumor cells by including an age variable. The underlying process for tumor growth is the same as classical tumor models, where growth is driven by pressure-limited cell proliferation, and cell movement away from regions of high pressure. 

The main contribution of this paper is establishing the convergence of solutions of the age-structured model to a limit satisfying a Hele-Shaw free boundary problem. This limiting problem describes the geometric motion of the tumor as it grows according to a nonlinear Darcy's law.

\end{abstract}
\maketitle
\section{Introduction}\label{sec:intro}
 Mathematical models for tumor growth have been extensively studied \cite{ Bellomo2008OnTF,david,Greenspan,PerthameQuiros2014}. These models complement
experimental findings and help further our understanding of cancer development and progression. Tumor
models can include various levels of complexity; in this paper we will focus on a simple model describing tumor growth as driven by two factors. The first is
cell proliferation (division), which is limited by internal pressure. The second is cell movement away from regions of high pressure, described by Darcy's law \cite{Greenspan,CL}. 

There are two commonly studied classes of tumor growth models. The first describes tumor evolution at the level of cell density. A classical model of this type was introduced in \cite{byrne}, and has been widely studied since, e.g., \cite{degond, PerthameQuiros2014,mellet2015heleshawproblemtumorgrowth}. This model introduces the cell density function, $\rho(t,x),$ which is governed by the porous media-type equation,
\begin{equation}\label{simple.tm}
\partial_t \rho-\div_x(\rho\nabla p)=\rho \Phi(p), \qquad p=\frac{m}{m-1}\Big(\frac{\rho}{\rho_M}\Big)^{m-1},
\end{equation}
where $p(t,x)$ represents pressure. The growth rate, $\Phi(p)$, is a decreasing function of pressure, satisfying $\Phi(p_M)=0$ for some maximum pressure $p_M$, referred to as the \textit{homeostatic pressure}, the pressure at which cell proliferation ceases \cite{basan09}. Pressure is caused by cells competing for space: for $m\gg 1$, the pressure is high when cell density $\rho>\rho_M$, and low when $\rho<\rho_M$. For simplicity of notation, we will assume from now on that the maximum packing density is $\rho_M=1$.

The second class of tumor growth models describes the geometric growth of the solid tumor via a free boundary problem, rather than describing cell density \cite{Cui07042008,Friedman2008}. In the seminal work of Perthame, Quir\'os, and V\'azquez, \cite{PerthameQuiros2014}, a link between the cell density model and the free boundary model is established. The authors show that the asymptotic limit of \eqref{simple.tm} as $m \rightarrow \infty$ yields a free boundary problem of Hele-Shaw type. This limit is referred to as an \textit{incompressible limit}, referencing the increasing ``stiffness" of the pressure as $m$ increases.

\subsection{Age-structured mechanical model} In this work, our objective is to establish a link between the cell density model and a free boundary problem for an age-structured tumor model related to \eqref{simple.tm}. In our previous work, \cite{llmw}, we introduce an age-structured model for tumor growth, extending \eqref{simple.tm} to take into account the life cycle of tumor cells. We associate an age $\theta\geq 0$ with each cell, representing the physiological age of the cell, i.e., how far along in the cell cycle it is. We consider a simplified cell cycle consisting of two phases: The first is interphase, where a cell grows and copies its genetic material, and the second is mitosis, where a cell splits into two daughter cells of age zero, and restarts the cycle. 

We introduce the cell distribution function, $n(t,\theta,x),$ which can be interpreted as the probability of finding a cell of age $\theta \geq 0$ at position $x\in \R^d$ and time $t \geq 0$. The evolution of this distribution function is described by the following boundary value problem:
\begin{equation}\label{prol.e1}
    \begin{cases}
        \partial_t n+r(p)\partial_\theta n-
    \text{div}_x(n \nabla_x p)=-r(p)\nu(\theta,p)n -\mu(\theta)n\qquad &t \geq 0,\; \theta \in (0, \infty),\; x \in \Omega, \\
        n(t,0,x)=2  \int_0^{\infty}\nu(\theta,p)n(t,\theta,x) d\theta \qquad &t \geq 0,\; x \in \Omega, \\
        n(0,\theta,x)=n_{0}(\theta,x)\qquad & \theta \geq 0, \;x \in \Omega, 
    \end{cases}
\end{equation}
where $\Omega \subset \R^d$ for $d \geq 1$. The volume density is given by: $$\rho(t,x)=\int_{0}^{\infty}V(\theta)n(t,\theta,x) d\theta,$$ where $V(\theta)$ is the volume of a cell of age $\theta$, and the pressure is given by $$p(t,x)=\frac{m}{m-1}\rho(t,x)^{m-1},$$ for a typical porous-media-type equation \cite{pme}. We assume that $\nu(\cdot,p)$ and $r(p)$ are decreasing, and $$r(p)=0, \nu(\cdot,p)=0 \text{ for all } p \geq p_M,$$ for some $p_M>0,$  as is natural for pressure-limited growth. 

In \eqref{prol.e1}, the growth term $r(p)$ represents the rate of cell growth (``aging"). In this setting, age is not linear in time and does not simply represent the time since last mitosis, but represents how far along in the cell cycle a cell has progressed. It is assumed that a cell can only progress in its cell cycle if the pressure is not too high: when $p=p_M,$ the cell will not age and will be stuck in its current age until pressure decreases. The function $\nu(\theta,p)$ is the rate of division of cells of age $\theta$ under local pressure $p$. When a cell enters its mitotic phase, the cell disappears (the term $-r\nu n$) and is replaced by two cells of age $\theta=0$ (the boundary condition in the second line of \eqref{prol.e1}). Cells will not divide when the pressure is too high, precisely, when $p=p_M.$ The term  $\mu(\theta)$ represents the age-dependent cell death rate. 

In \cite{llmw}, it is shown that weak solutions of \eqref{prol.e1} exist, and the behaviour is analyzed numerically. These results are discussed further in Section \ref{sec:exist}. The motivation behind studying this age-structured model is that cells at different phases of
their life cycle duplicate and die at different rates. Solutions encode the age
distribution of cells, providing insight into where within the tumor most proliferation occurs. Such knowledge of the tumor’s spatial structure can lead to more accurate predictions and help with therapy development \cite{lewin,noble}.  Therapy development was one of the main motivations of this work: cells at different stages of their life-cycle respond differently to treatments, so understanding the age-structure helps in designing effective therapies to slow tumor growth. 

 The main contribution of this work is establishing convergence of solutions of \eqref{prol.e1}, as $m \rightarrow \infty$, to a limit solving a Hele-Shaw-type free boundary problem. (The precise statements can be found in Section \ref{sec:incomp}). 

\subsection{Motivations and Related Work}
The model, \eqref{prol.e1}, discussed in this paper is motivated by previous work on space-homogeneous age-structured models. The main result of this paper, the incompressible limit of \eqref{prol.e1}, is motivated by previous work establishing the convergence of porous media-type tumor models to incompressible limits, e.g., \cite{PerthameQuiros2014,david, mellet2015heleshawproblemtumorgrowth,kim2022,david2020freeboundarylimittumor}.

\medskip
\noindent \textbf{Age-structured models for tumor growth.}
Age-structured tumor models have been of interest for some time, motivated by the fact that cell proliferation and death rates depend on the age of the cell \cite{DYSON200273, GGQ12, GYLLENBERG198767}. However, most previous models lack spatial dependence, and so do not consider the effect of pressure on proliferation rate or the effect of volume change as cells age. A simple model of this type, known as the growth-fragmentation equation, which is the space-homogenous version of \eqref{prol.e1}, is given by, 
\begin{align}\label{s.hom}
    \begin{cases}
        \partial_t n +\partial_\theta n =-\beta(\theta)n-\mu(\theta)n, \qquad &t>0,\theta>0, \\
        n(t,0)=2\int_{0}^{\infty}\beta(\theta)n(t,\theta)d\theta ,\qquad &t>0,\\
        n(0,\theta)=n_0(\theta)\qquad &\theta>0.
    \end{cases}
\end{align}
 
In \cite{llmw}, we incorporate space-dependence into the age-structure model, combining the ideas of the cell density model and the growth-fragmentation model. 

\medskip
\noindent \textbf{Incompressible limit.}
The theory surrounding the convergence of porous media equations without a growth term, i.e., $\partial_t \rho -\Delta \rho^m=0$, is well developed, and it is known that the limit as $m \rightarrow \infty$ approaches the standard Hele-Shaw free boundary problem \cite{benilin,elliott}.  When $1<m<\infty,$ the equation for $\rho$ is of porous media-type, meaning the density spreads with a finite speed of propagation \cite{pme}. It thus seems natural that the asymptotic limit would satisfy a free-boundary problem.

In the work of Perthame, Quir\'os, and V\'azquez, \cite{PerthameQuiros2014}, this result was extended to tumor models of the type \eqref{simple.tm}, which are porous medium equations with a nonzero right hand side. The authors show that the solution $(\rho_m,p_m)$ of the model \eqref{simple.tm} converges to an incompressible limit as $m \rightarrow \infty$, which solves a Hele-Shaw free-boundary problem. Precisely, the limit $(\rho_\infty,p_\infty)$  satisfies the following \textit{Hele-Shaw graph}:
\begin{equation}\label{hs1}
p_\infty \in
    \begin{cases}
        0, \quad 0\leq \rho_\infty <1, \\
        [0,\infty), \quad \rho_\infty =1,
    \end{cases}
\end{equation}
and the following \textit{complementarity formula}:
$$
p_\infty(\Delta p_\infty +\Phi(p_\infty)), \quad \text{ in } \mathcal{D}'([0,T]\times \R^d),
$$
in the sense that
\begin{equation*}
    \int_{\R^d}\varphi\bigl(-|\nabla p_\infty|^2+p_\infty \Phi(p_\infty)\bigr)-p\nabla \varphi \cdot \nabla p \; dx=0, 
\end{equation*}
for any test function $\varphi \in C^1([0,T]\times \R^d)$.
This gives rise to a free-boundary problem, where the region $$
\Omega(t):=\{x: p_\infty(t,x)>0\}$$ can be thought of as the ``tumor". The authors show that the normal velocity $V$ at the boundary of the tumor $\Omega(t)$ is given by 
$$
V=|\nabla p_\infty|,
$$
following Darcy's law. Further, the uniqueness of the limit $(\rho_\infty,p_\infty)$ is shown by Hilbert's uniqueness method \cite{Bnilan1982SolutionsOT,crowley}. Since then, the incompressible limit of porous media-type tumor growth models has been studied heavily, e.g., \cite{david2020freeboundarylimittumor,mellet2015heleshawproblemtumorgrowth,kim2022,Guillen_2022}.

This work has been extended to cross-diffusion equations, i.e., systems that consider multiple cell populations. These systems can describe various interactions such as cancer cells with healthy cells, normal cells with self-destroying (autophagic) cells, or cells with varying phenotypes \cite{BubbaPerthamePoucholSchmidtchen2020,david,liuxu,PriceXu2020}.

In the more recent work of David, \cite{david}, the author introduces a tumor growth model structured by \textit{phenotypic traits}. The structured model is a cross-diffusion model describing an infinite number of  interacting cell populations. The density, $n(t,y,x)$, of a cell with phenotype $y \in [0,1]$, is described by the following system:
\begin{equation}\label{pheno}
\begin{cases}
    \partial_t n_m(t,y,x) -\nabla (n_m(t,y,x) \nabla p_m(t,x))=n_m(t,y,x)R(y,p_m(t,x)), \qquad (t,y,x)\in (0,T]\times[0,1]\times\R^d, \\
    \rho_m(t,x)=\int_{0}^{1}n_m(t,y,x)dy, \\
    p_m(t,x)=\rho_m(t,x)^m.
    \end{cases}
\end{equation}
This model is structured by a phenotypic trait which is assigned to each cell at birth and does not change over time. David shows that, under appropriate assumptions on $R$ and the initial data, as $m \rightarrow \infty$, the solution $(n_m,\rho_m,p_m)$ of the model \eqref{pheno} approaches an incompressible limit $(n_\infty,\rho_\infty,p_\infty)$, which solves a Hele-Shaw free-boundary problem. In particular, $(\rho_\infty,p_\infty)$ satisfy the Hele-Shaw graph \eqref{hs1}, and the pair $(n_\infty,p_\infty)$ satisfies the following complementarity formula:
\begin{equation*}
    p_\infty\Bigl(\Delta p_\infty +\int_{0}^{1}n_\infty R(y,p_\infty)\;dy \Bigr)=0 \text{ in }\mathcal{D}'([0,T]\times \R^d).
\end{equation*}

The structured model, \eqref{pheno}, is a strong motivation for our current work. In this paper, we follow the outline of \cite{david} to establish an incompressible limit of \eqref{prol.e1}. The main point where our proof diverges from that of \cite{david} is in establishing the strong convergence of $\nabla p_m$. In the age-structured case, the addition of the transport term in $\theta$ means that $\theta$ cannot be treated as a parameter, as $y$ could. Another difference is that $\theta \in [0,\infty),$ an unbounded domain, rather than $y \in [0,1].$ These differences cause a challenge in proving strong convergence. The key to overcoming this obstacle is to show that $n_m$ is  compactly supported in $x$ and $\theta$, uniformly in $m$. Additionally, this work extends \cite{david} by incorporating an age-dependent volume, meaning $\rho$ depends not only on the number of cells ($\rho=\int_{0}^{1}n\,dy$ in \eqref{pheno}), but instead on the volume they take up ($\rho=\int_{0}^{\infty}V(\theta)n\,d\theta$ in \eqref{prol.e1}).

\medskip
\noindent \textbf{Main Contribution.} In this paper, we  follow the outline of \cite{david} and \cite{PerthameQuiros2014} to show convergence of solutions to \eqref{prol.e1} to an incompressible limit, which satisfies Hele-Shaw free boundary problem and a complementarity formula.

\subsection{Definitions and notation}

\medskip

We will define the equation describing the evolution of the density, $\rho.$
Multiplying \eqref{prol.e1} by $V(\theta)$ and integrating in $\theta$ yields the following equation for $\rho$:
\begin{equation*}
\begin{cases}
    \partial_t \rho -\text{div}(\rho\nabla p)=\int_{0}^{\infty}n\Bigl(r(p)[V'(\theta)+(2V(0)-V(\theta))\nu(\theta,p)]-\mu(\theta)V(\theta)\Bigr)\;d\theta\qquad & t \in [0,T],\;x \in \R^d
    \\
    \rho(0,x)=\rho_{0}(x)=\int_{0}^{\infty}V(\theta)n_{0}(\theta,x)\,d\theta \qquad & x \in \R^d.
    \end{cases}
\end{equation*}
For simplicity of notation, we define
\begin{equation}\label{capf}
F(\theta,p):=r(p)[V'(\theta)+(2V(0)-V(\theta))\nu(\theta,p)]-\mu(\theta)V(\theta),
\end{equation}
so the equation for $\rho$ becomes
\begin{equation}\label{prol.e2}
\begin{cases}
    \partial_t \rho -\text{div}(\rho\nabla p)=\int_{0}^{\infty}n(t,\theta,x)F(\theta,p)\;d\theta\qquad & t \in [0,T],\;x \in \R^d
    \\
    \rho(0,x)=\rho_{0}(x)=\int_{0}^{\infty}V(\theta)n_{0}(\theta,x)\,d\theta \qquad & x \in \R^d.
    \end{cases}
\end{equation}

From the assumptions on $r$ and $\nu$, the function $F(\cdot, p)\leq 0$ for $p \geq p_M$. We  highlight two particular cases of $F(\theta,p)$, which have different biological meanings and result in very different dynamics. 

Case 1: Consider cells that have constant volume, that is, $V(\theta)=V_0$ for all $\theta\geq 0,$  Then, $$F(\theta,p)=V_0r(p)\nu(\theta,p)-\mu(\theta)V_0,$$ and it is apparent that tumor growth is driven primarily by cell division. If we assume further that $\nu(\theta,p)=\nu(p),$ and $\mu(\theta)=\mu_0$, the equation on $\rho$ exactly recovers the original porous media tumor growth model, \eqref{simple.tm}, with $\Phi(p)=r(p)\nu(p)V_0-\mu_0V_0.$

Case 2: Consider cells with age-dependent volume, and assume that mitosis is volume-preserving, that is, the total size of two daughter cells is equal to the size of the parent cell. This can be done by imposing that 
$$
\nu(\theta,p)(2V(0)-V(\theta))=0,
$$
for all $\theta\geq 0.$ In this case, 
$$
F(\theta,p)=r(p)V'(\theta)-\mu(\theta)V(\theta).
$$ It can be seen that tumor growth is now driven by the increase in volume of individual cells prior to cell division.

\subsection{Outline of the paper} In Section \ref{sec:statement}, we will discuss assumptions and state the main results of the paper. In Section \ref{sec:exist}, we recall the existence results for our model from \cite{llmw}. In Section \ref{sec:incomp}, we prove the convergence to an incompressible limit as well as the complementarity formula for the limit. 
\section{Assumptions and Main Results}\label{sec:statement}
Given $m>2$, we consider $n_m$, the weak solution of the following boundary value problem (see Theorem \ref{thm:1} for a precise definition of weak solution).

\begin{equation}\label{nm}
    \begin{cases}
        \partial_t n_m+r(p_m)\partial_\theta n_m-\text{div}(n_m \nabla p_m)=-r(p_m)\nu(\theta,p_m)n_m-\mu(\theta)n_m \qquad &t \in [0,T],\; \theta \geq 0,\; x \in \R^d,\\
        n_m(t,0,x)=2\int_{0}^{\infty}\nu(\theta,p_m) n_m(t,\theta,x)\,d\theta \qquad &t \in [0,T],\; x \in \R^d, \\
        n_m(0,\theta,x)=n_{0,m}(\theta,x) \qquad &\theta\geq 0,\; x \in \R^d,
    \end{cases}
\end{equation}
where $p_m(t,x)=\frac{m}{m-1}\rho_m(t,x)^{m-1}.$
We make the following assumptions on the various parameters of the problem:

\begin{equation}\label{a.beta}
0 \leq \mu(\theta)\leq \mu_{max}, \quad |\mu'(\theta)|\leq C,
\end{equation} 
\begin{equation}\label{a.r}
0\leq r(p)\leq r_{max}, \quad r'(p)\leq 0, \quad r(p_M)=0 ,
\end{equation}
\begin{equation}\label{a.nu}\nu \in C^1([0,\infty)\times[0,p_M]), \text{ }0 \leq\nu(\theta,p) \leq C, \text{ } \partial_p \nu(\theta,p)\leq 0, \text{ } |\partial_\theta \nu(\theta,p)|\leq C, \text{ } \nu(\theta,p_M)=0,  \end{equation} 
 and
\begin{equation}\label{v.assum}  
V \in C^2([0,\infty)), \quad 0\leq V'(\theta)\leq C, \quad |V''|\leq C, \quad \text{and} \quad V_0:=V(0)\leq V(\theta)\leq 2V_0,
\end{equation}
for constants $\mu_{max}, r_{max},p_M, V_0>0,$ and various constants all called $C>0$.
Multiplying \eqref{nm} by $V(\theta)$ and integrating in $\theta$ yields that $\rho_m(t,x)=\int_{0}^{\infty}V(\theta)n_m(t,\theta,x)\,d\theta$ solves
\begin{equation}\label{rho.comp}
\begin{cases}
    \partial_t \rho_m -\text{div}(\rho_m\nabla p_m)=\int_{0}^{\infty}n_mF(\theta,p_m)\;d\theta\qquad & t \in [0,T],\;x \in \R^d
    \\
    \rho_m(0,x)=\rho_{0,m}(x)=\int_{0}^{\infty}V(\theta)n_{0,m}(\theta,x)\,d\theta \qquad & x \in \R^d,
    \end{cases}
\end{equation}
where $F$ is defined by \eqref{capf}.
We make the following assumptions on the initial data. Assume that
\begin{equation}\label{assumpinit}
    n_{0,m} \in L^1\cap L \log L ( (0,\infty)\times \R^d), \quad (|x|^2+\theta) n_{0,m} \in L^1 ( (0,\infty)\times \R^d), \quad 0\leq \rho_{0,m} \leq (\frac{m-1}{m}p_M)^{\frac{1}{m-1}},
\end{equation}
where the bounds are independent of $m$.
Assume there exists $n_{0}\in L^1([0,\infty)\times\R^d)$ such that \begin{equation}\label{initconv}||n_{0,m} - n_{0}||_{L^1([0,\infty)\times \R^d)}\xrightarrow[m\rightarrow \infty]{} 0.\end{equation}

In this paper, we show that, up to a subsequence, as $m \rightarrow \infty$, $n_m$, $\rho_m$, $p_m$ converge to limits $n_\infty$, $\rho_\infty$, and $p_\infty$, solving the following equation:
\begin{equation}\label{n.lim}
\begin{cases}
    \partial_t n_\infty +r(p_\infty)\partial_\theta n_\infty -\text{div}(n_\infty \nabla p_\infty)=-r(p_\infty)\nu(\theta,p_\infty)n_\infty-\mu(\theta)n_\infty\qquad \text{ in } \mathcal{D}'([0,T]\times(0,\infty)\times \R^d)\\
        n_\infty(t,0,x)=2\int_{0}^{\infty}\nu(\theta,p_\infty)n_\infty(t,\theta,x)\,d\theta, \\
        n_\infty(0,\theta,x)=n_{0}(\theta,x),
\end{cases}
\end{equation}
where 
\begin{equation}\label{rhoinf}
\rho_\infty(t,x)=\int_{0}^{\infty}V(\theta)n_\infty(t,\theta,x)d\theta,
\end{equation}
and $\rho_\infty,p_\infty$ satisfy the Hele-Shaw condition, 

\begin{equation}\label{hs}
p_\infty\in\begin{cases}0, \qquad &0 \leq \rho_\infty <1, \\
[0,\infty), \qquad &\rho_\infty=1.
\end{cases}
\end{equation}
Further, $\rho_\infty$ solves
\begin{equation}\label{rho.lim}
\begin{cases}
\partial_t\rho_\infty -\Delta p_\infty =\int_{0}^{\infty}n_\infty F(\theta,p_\infty))\,d\theta\qquad & \text{ in } \mathcal{D}'([0,T]\times \R^d),
    \\
    \rho_\infty(0,x)=\rho_{0}(x)=\int_{0}^{\infty}V(\theta)n_{0}(\theta,x)\,d\theta.
    \end{cases}
\end{equation}

We now state the two main convergence results.
\begin{theorem}[Weak Free-Boundary Limit]\label{prop.conv}
    Let $(n_m,p_m,\rho_m)$ be a weak solution of \eqref{nm}, under the assumptions \eqref{a.beta}, \eqref{a.r}, \eqref{a.nu}, and \eqref{v.assum}. Assume that the initial data satisfies \eqref{assumpinit}-\eqref{initconv}. Further, assume that $n_{0,m}$ is compactly supported in $x$ and $\theta$, uniformly in $m$. Then, up to a subsequence, as $m \rightarrow \infty$, $p_m$ converges strongly in $L^2([0,T]\times \R^d)$, $\rho_m$ converges weakly in $L^\infty(0,T;L^q(\R^d))$ and $n_m$ converges weakly-* in $L^\infty(0,T;L^q(\R^d;L^1([0,\infty))))$, for all $q \in [1,\infty]$, to limits $p_\infty,\rho_\infty,n_\infty$ satisfying \eqref{n.lim}-\eqref{rho.lim}.
 \end{theorem}
 \begin{remark}\label{weakstar} In Theorem \ref{prop.conv}, $n_m$ converges weakly-* in $L^\infty(0,T;L^q(\R^d);L^1([0,\infty)))$. This means that for any $\varphi \in L^1(0,T;L^{q'}(\R^d;L^\infty([0,\infty))),$ up to a subsequence,
    $$
    \int_{0}^{T}\int_{\R^d}\int_{0}^{\infty}n_m(t,\theta,x)\varphi(t,\theta,x)d\theta dx dt \rightarrow \int_{0}^{T}\int_{\R^d}\int_{0}^{\infty}n_\infty(t,\theta,x)\varphi(t,\theta,x)d\theta dx dt,
    $$
    for some $n_\infty \in L^\infty(0,T;L^{q}(\R^d;L^1([0,\infty))).$
    \end{remark}
 \begin{theorem}[Complementarity Formula]\label{complement}
     The limits $p_\infty$, $n_\infty$ satisfy the relation
     \begin{equation}\label{comp}
        p_\infty\Bigl(\Delta p_\infty+\int_{0}^{\infty}n_\infty\bigl(F(\theta,p_\infty)-\mu(\theta)V(\theta)\bigr)\,d\theta \Bigr)=0 \text{ in } \mathcal{D}'([0,T]\times \R^d).
     \end{equation}
 \end{theorem}
 These results will be proved in Section \ref{sec:incomp} using several lemmas that establish the convergence of $p_m, \rho_m,$ and $n_m$.
 
\section{Existence of solutions to the model}\label{sec:exist}
In this section, we recall the existence of solutions to the model \eqref{prol.e1}, as proved in our previous work \cite{llmw}. A slightly simplified model was considered:
\begin{equation}\label{simpler.n}
 \begin{cases}
        \partial_t n+\partial_\theta n-
    \text{div}(n \nabla p)=-\nu(\theta,p)n \qquad &t \geq 0,\; \theta \in (0, \infty),\; x \in \R^d, \\
        n(t,0,x)=2  \int_0^{\infty}\nu(\theta,p)n(t,\theta,x) d\theta \qquad &t \geq 0,\; x \in \R^d, \\
        n(0,\theta,x)=n_{0}(\theta,x)\qquad & \theta \geq 0, \;x \in \R^d ,
    \end{cases}
\end{equation}
with $\rho(t,x)=\int_{0}^{\infty}V(\theta)n(t,\theta,x)d\theta,$ and the same assumptions on $V$ and $\nu$ as stated in Section \ref{sec:intro}. It is remarked in \cite{llmw} that extending the existence results for \eqref{simpler.n} to the full model \eqref{prol.e1} is not difficult. We state below the main existence result:
\begin{theorem}[\cite{llmw}, Theorem 2.1]\label{thm:1}
For all $m>2$ and for any initial condition $n_{0}(\theta,x ) $ such that
\begin{equation}\label{eq:init}
n_{0} \in L^1\cap L \log L ( (0,\infty)\times \R^d), \quad (|x|^2+\theta) n_{0} \in L^1 ( (0,\infty)\times \R^d), \quad \rho_{0} \in L^\infty(\R^d),
\end{equation}
there exists a weak solution $n(t,\theta,x) $ of \eqref{simpler.n}.
The function $n(t,\theta,x)$ is such that
$$
n \in L^\infty(0,T;L^1\cap L \log L ( (0,\infty)\times \R^d)), \quad (|x|^2+\theta) n \in L^\infty(0,T; L^1 ((0,\infty)\times \R^d)),
$$
while the density $\rho(t,x)$ and pressure $p(t,x) = \frac m{m-1} \rho(t,x)^{m-1}$ are such that
$$
\rho \in L^\infty(0,T; L^1\cap L^\infty(\R^d)), \qquad \na p \in L^2(0,T; L^2(\R^d)),
$$
and the equation \eqref{simpler.n} is satisfied in the following distributional sense:
\begin{equation}\label{eq:weak}
\begin{split}
& \int_0^T\!\!\!\int_{\R^d}\! \int_0^\infty n(t,\theta,x)\big[ -\pa_t \psi (t,\theta,x)- \pa_\theta\psi (t,\theta,x)+ \na_x p(t,x) \cdot\na_x\psi (t,\theta,x) + \nu(\theta,p(t,x))\psi (t,\theta,x)\big]\, d\theta\, dx\, dt\\
& \qquad = \int_{\R^d}\!\int_0^\infty n_{0} (\theta,x) \psi(0,\theta,x)\, d\theta\, dx
+ \int_0^T \!\!\! \int_{\R^d} \psi(t,0,x) 2 \int_0^\infty \nu(\theta,p(t,x)) n(t,\theta,x)\, d\theta\, dx\, dt,
\end{split}
\end{equation}
for all $\psi \in \mathcal D([0,T)\times [0,\infty)\times \R^d)$. Furthermore, the density $\rho(x,t)$ solves 
\begin{equation}\label{eq:mun0}
\begin{split}
 \pa_t \rho - \div(\rho \na p)  = \int_0^\infty  V'(\theta) n(t,\theta,x)\, d\theta 
 +  \int_0^\infty \nu(\theta,p)(2V(0)-V(\theta)) n(t,\theta,x)\, d\theta .
\end{split}
\end{equation}
\end{theorem}
In this paper, we use the same assumptions as stated in Theorem \ref{thm:1}, so as to ensure existence.

As noted in \cite{llmw}, the uniqueness for  cross-diffusion models such as \eqref{prol.e1} is 
a difficult and largely open problem, and was not addressed in our previous work nor this current work.

In \cite{llmw}, numerical approximations were performed to examine the behaviour of solutions to \eqref{simpler.n}, especially to keep track of where in the tumor most of the cell proliferation takes place. Numerical simulations showed that most proliferation occured in an annulus near the outer edge of the tumor, and very little occurred on the inside of the tumor. It was observed that the inner part of the tumor was mostly very old cells, which can be interpreted as a ``necrotic core". These results are consistent with in-vitro experiments on cancer cells, which exhibit a necrotic core and a proliferating rim \cite{byrne,thom,noble,lewin}. 

These numerical findings highlight the importance of including the spatial dependence in age-structure models. Considering the spatial structure of tumors, that is, where in the tumor most
proliferation is occurring, is important when designing mathematical models because this knowledge can assist
with clinical forecasting and treatment development.
\section{Incompressible Limit}\label{sec:incomp}

In order to prove Theorems \ref{prop.conv} and \ref{complement}, we will prove several lemmas which establish a priori bounds on $\rho_m,p_m$ and $n_m$.

We begin by remarking that $p_m$ solves  \begin{equation}\label{p.eq}
        \partial_t p_m=(m-1)p_m\Delta p_m+|\nabla p_m|^2+m(\frac{m-1}{m}p_m)^{\frac{m-2}{m-1}}\int_{0}^{\infty}n_mF(\theta,p_m)\, d\theta,
    \end{equation}
    and by recalling the definition of $F$, \eqref{capf}, and the assumptions on $V$ and $\nu$, it holds that $F(\theta,p)\leq Cr(p)$, for a constant $C$ independent of $m$. Then,
    \begin{equation}\label{p.eqbd}
     \partial_t p_m\leq (m-1)p_m\Delta p_m+|\nabla p_m|^2+C(m-1)p_mr(p_m).
    \end{equation}
 
\begin{lemma}
    Assume that $\rho_{0,m}\geq 0$ and $\frac{m}{m-1}\rho_{0,m}^{m-1}\leq p_M.$ Then $\rho_m$ and $p_m$ are bounded in $L^\infty([0,T]\times\R^d)$ uniformly in $m$, that is, for every $m \geq 3$,
    \begin{equation*}
        0\leq p_m \leq p_M \text{ and } 0\leq \rho_m \leq (\frac{m-1}{m}p_M)^{1/(m-1)}:=\rho_{m,M}\leq \max\{p_M,1\}.
    \end{equation*}
\end{lemma}
\begin{proof}
   
   From the assumptions made on $V, r, \mu,$ and $\nu$, we have that $|F(\theta,p_m)|\leq C$ for a constant $C$ only depending on the bounds of $V$, $r$, $\mu$, and $\nu$. From \eqref{p.eq} and \eqref{p.eqbd}, the assumptions on $\rho_{0,m}$, and the fact that $r(p_M)=0$, $0$ is a subsolution and $p_M$ is a supersolution of \eqref{p.eq}. 
    By the comparison principle, the desired bound on $p_m$ holds, i.e., $0\leq p_m \leq p_M$. The desired bound on $\rho_m$ then follows from the definition of $p_m$. 
\end{proof}
\begin{lemma}
    Assuming $\rho_{0,m}\in L^1(\R^d),$ bounded uniformly in $m$, then $\rho_m, p_m$ are bounded in $L^\infty(0,T;L^1(\R^d))$ uniformly in $m$, and $n_m$ is bounded in $L^\infty(0,T;L^1([0,\infty)\times \R^d))$ uniformly in $m$.
\end{lemma}
\begin{proof}
    Integrating \eqref{rho.comp} with respect to $x$,
    \begin{align*}
        \frac{d}{dt}\int_{\R^d}\rho_m(t,x) \,dx&=\int_{\R^d}\int_{0}^{\infty}n_mF(\theta,p_m)\, d\theta dx\\
        & \leq C\int_{\R^d}\rho_m(t,x)\,dx,
    \end{align*}
    for a $C$ depending only on the bounds of $V, V', \nu$, and $r$, from the definition of $F$.
    By Gronwall's inequality,
    \begin{equation*}
        \int_{\R^d}\rho_m(t,x)\, dx \leq e^{Ct}\int_{\R^d}\rho_{0,m}(x)\, dx,
    \end{equation*}
    and so 
   
    \begin{equation*}
        \sup_{t \in[0,T]}\int_{\R^d}\rho_m(t,x)\,dx dt \leq C(T).
    \end{equation*}
    For the bound on $p_m$,
    $$
    p_m=\rho_m\frac{m}{m-1}(\frac{m-1}{m}p_m)^{\frac{m-2}{m-1}},
    $$
    so $$
    \int_{\R^d}p_m\,dx \leq \frac{m}{m-1}(\frac{m-1}{m})^\frac{m-2}{m-1}p_M^{\frac{m-2}{m-1}}\int_{\R^d}\rho_m(t,x)\,dx \leq \tilde{C}e^{Ct}\int_{\R^d}\rho_{0,m}\,dx,$$
    for $\tilde{C}$ depending only on $p_M,$
    and so,
    $$
    \sup_{t\in[0,T]}\int_{\R^d}p_m(t,x)\,dx \leq C(T).
    $$
    For the bound on $n_m$,
  
    \begin{align*}
        \sup_{t\in[0,T]}\int_{0}^{\infty}\int_{\R^d}n_m(t,\theta,x)\,dxd\theta&\leq\sup_{t\in [0,T]}\frac{1}{V_0}\int_{0}^{\infty}\int_{\R^d}V(\theta)n_m(t,\theta,x)dx d\theta
        \\ &=\sup_{t\in[0,T]}\frac{1}{V_0}\int_{\R^d}\rho_m(t,x)\,dx\\ &\leq C(T).
    \end{align*}
   
\end{proof}

\begin{lemma}\label{q.lemma}
    For any $q\geq1$, $\rho_m$ is bounded in $L^\infty(0,T;L^q(\R^d))$.
\end{lemma}
\begin{proof}
    By multiplying \eqref{rho.comp} by $\rho^{q-1}$ and integrating in $x$ yields,
    \begin{equation}\label{q.norm}
        \frac{d}{dt}\int_{\R^d}\rho_m^q \, dx +mq(q-1)\int_{\R^d}\rho_m^{m+q-3}|\nabla \rho_m|^2 \, dx \leq Cq\int_{\R^d}\rho_m^q\, dx.
    \end{equation}
    By Gronwall's inequality and the non-negativity of $\rho_m,$
    \begin{equation*}
        \int_{\R^d}\rho_m^q \,dx\leq e^{Cqt}\int_{\R^d}\rho_{0,m}^q\, dx.
    \end{equation*}
\end{proof}
\begin{remark}
    It is clear from the proofs of the previous two lemmas that $\rho_m$, $p_m$, and $\rho_m^q,$ $q \geq 1$, are also bounded in $L^1([0,T]\times \R^d).$
\end{remark}
We now state two lemmas that demonstrate the compact support of solutions in both space $x$ and age $\theta$. Although we are working on unbounded domains $\R^d$ in space and $[0,\infty)$ in age, after showing that solutions have compact support, we can restrict our argument to bounded domains. The compact support in space is due to the assumptions that $\rho_{0,m}$ is compactly supported uniformly in $m$, and the finite speed of propagation of porous medium equations \cite{pme}. The compact support in age is due to the assumption that $n_{0,m}$ is compactly supported in $\theta$ uniformly in $m$.
\begin{lemma}[\cite{PerthameQuiros2014}, Lemma 2.6]\label{lem.cptspt}
    Assume that $\rho_{0,m}$ is compactly supported uniformly in $m$, that is, all of their supports are contained in a ball $B_R$. Then for all $T>0$, there exists a constant $C_T,$ depending only on $T$ and $R$, such that for any $|x|\geq C_T,$ $\rho_m(t,x)=0$ for all $m>2$ and for all $t \in [0,T].$
\end{lemma}
\begin{proof}
    From \eqref{rho.comp}, we have that
    $$
    \partial_t \rho_m -\div(\rho_m \nabla p_m)\leq C\rho_m,
    $$
    where $C$ is a constant depending only on the bounds on $\nu$, $r$, $V$ and $V'$. It holds that $\rho_m$ is a subsolution of the equation considered in \cite{PerthameQuiros2014}, so by Lemma 2.6 in that work, the solutions $\rho_m$ are compactly supported in space, uniformly in $m$.
\end{proof}
\begin{lemma}\label{t.cptspt}
    Assume that $n_{0,m}$ is compactly supported in $\theta$ uniformly in $m$, that is, there exists $\theta_{in}>0$ such that for all $m>2$, $n_{0,m}(\theta,x)=0$ for all $\theta \geq \theta_0$. Then, for all $T>0$, $n_m(t,\theta,x)=0$ for any $\theta \geq \theta_{in}+r_{max}T,$ for all $m>2$ and for all $t \in [0,T],$ $x \in \R^d.$
\end{lemma}
\begin{proof}
    Assume that
    \begin{equation}\label{a.cpttheta}
    \int_{\R^d}\int_{\theta_{in}}^{\infty}n_{0,m}(\theta,x)d\theta dx=0.
    \end{equation}
    Consider 
    $$
   \int_{\R^d} \widetilde{\rho_m}(t,x)\,dx:=\int_{\R^d}\int_{\theta_{in}+r_{max}t}^{\infty}n_m(t,\theta,x)\,d\theta \, dx,
    $$
    where $r_{max}$ is from \eqref{a.r}.  Then $\int_{\R^d}\tilde{\rho_m}(0,x)\,dx=0$ from \eqref{a.cpttheta}. 
    Taking the derivative in $t$,
    \begin{align*}
       \partial_t \int_{\R^d}\widetilde{\rho_m}&= \partial_t \int_{\R^d}\int_{\theta_{in}+r_{max}t}^{\infty}n_m(t,\theta,x)d\theta dx \\&=\int_{\R^d}\int_{\theta_{in}+r_{max}t} \partial_t n_m(t,\theta,x)d\theta dx -\int_{\R^d}n_m(t,\theta_{in}+r_{max}t,x)r_{max}dx\\
       &\leq \int_{\R^d}\int_{\theta_{in}+r_{max}t}^{\infty}-r(p_m)\partial_\theta n_m +\div(n_m \nabla p_m)d\theta -n_m(t,\theta_{in}+r_{max}t,x)r_{max}dx \\
       &=\int_{\R^d}(r(p_m)-r_{max})n_m(t,\theta_{in}+r_{max}t,x)+\div(\tilde{\rho_m}\nabla p_m) dx \\
       & \leq 0,
    \end{align*}
    meaning that $\int_{\R^d}\widetilde{\rho_m}(t,x)dx=0$ for all $t\geq 0$. It follows directly that 
    $$
    \int_{\R^d}\int_{\theta_{in}+r_{max}T}^{\infty}n_m(t,x,\theta)\,d\theta \,dx=0,
    $$
    for all $t\geq 0$ and for all $m$.
\end{proof}

\begin{lemma}\label{secondbds}
    Assume $p_{0,m} \in L^1(\R^d).$ For $m \geq 3$, $p_m$ is bounded in $L^2(0,T;H^1(\R^d))$ and $\partial_t \rho_m $ is bounded in $ L^2(0,T;H^{-1}(\R^d))$, uniformly in $m$. Further, by defining $v_m=\rho_m^m$, $\nabla v_m$ is bounded in $L^2([0,T]\times \R^d)$ uniformly in $m$.
\end{lemma}
\begin{proof}
    Taking $q=m-1$ in \eqref{q.norm}, yields
    \begin{equation*}
        \frac{d}{dt}\int_{\R^d}p_m\, dx + (m-2)\int_{\R^d}|\nabla p_m|^2\, dx \leq C(m-1)\int_{\R^d}p_m\, dx.
    \end{equation*}
    Then, integrating from $0$ to $T$,
    \begin{equation*}
        \int_{\R^d}p_m(T,x)\,dx -\int_{\R^d}p_m(0,x)\,dx +(m-2)\int_{0}^{T}\int_{\R^d}|\nabla p_m|^2\, dx dt \leq C(m-1)\int_{0}^{T}\int_{\R^d}p_m\, dxdt,
    \end{equation*}
    and by the non-negativity of $p_m$,
    \begin{equation*}
        \int_{0}^{T}\int_{\R^d}|\nabla p_m|^2\, dx dt \leq C\frac{m-1}{m-2}\int_{0}^{T}\int_{\R^d}p_m\, dxdt+\frac{1}{m-2}\int_{\R^d}p_m(0,x)\,dx,
    \end{equation*}
    so $p_m \in L^2(0,T;H^1(\R^d))$. It follows that $\partial_t \rho_m \in L^2(0,T;H^{-1}(\R^d)).$
    Finally, 
    \begin{equation*}
        \nabla v_m=\rho_m\nabla p_m \in L^2([0,T]\times \R^d).
    \end{equation*}
\end{proof}
\begin{remark}
    Recall that
    $$
    \int_{0}^{\infty}|n_m(t,\theta,x)|d\theta \leq \frac{1}{V_0}\rho_m(t,x),
    $$
    and from the bounds on $\rho_m,$ it holds that
    $\int_{0}^{\infty}n_m d\theta$ is bounded uniformly in $m$ in $L^\infty(0,T;L^q(\R^d))$ for all $q \in [1,\infty].$ It follows that $n_m$ converges weakly-* in $L^\infty(0,T;L^q(\R^d;\mathcal{M}([0,\infty)))),$ for all $q \in [1,\infty]$, where $\mathcal{M}$ is the space of radon measures, in the following sense:
    
    For any $\varphi \in L^1(0,T;L^{q'}(\R^d;C^0([0,\infty)))),$ $$
    \int_{0}^{T}\int_{\R^d}\int_{0}^{\infty}n_m(t,\theta,x)\varphi(t,\theta,x)d\theta dx dt \rightarrow \int_{0}^{T}\int_{\R^d}\int_{0}^{\infty}n_\infty(t,\theta,x)\varphi(t,\theta,x)d\theta dx dt,
    $$
    for $n_\infty \in L^\infty(0,T;L^q(\R^d;\mathcal{M}([0,\infty)))).$
\end{remark}

We will show in Lemma \ref{lem.nlogn} that a stronger result holds, and in fact, $n_m$ converges weakly-* in $L^\infty(0,T;L^q(\R^d;L^1([0,\infty)))),$ meaning that $n_\infty$ is a function which is $L^1$ in $\theta$, not merely a measure.

From the above lemmas, it follows that there exist limit functions $\rho_\infty, p_\infty,v_\infty,n_\infty$ such that, up to a subsequence:
\begin{equation*}
    \begin{cases}
        \rho_m \rightharpoonup \rho_\infty \text{ weakly-* in $L^\infty(0,T;L^q( \R^d$))}\text{ for all } q \in [1,\infty],\\
        p_m \rightharpoonup p_\infty \text{ weakly-* in } L^\infty(0,T; L^q(\R^d)) \text{ for all } q \in [1,\infty],\\
        n_m\rightharpoonup n_\infty \text{ weakly-* in } L^\infty(0,T;L^q(\R^d;\mathcal{M}([0,\infty)))), \text{ for all } q \in [1,\infty], \\
        \nabla p_m \rightharpoonup \nabla p_\infty \text{ weakly in } L^2([0,T]\times \R^d),\\
        \nabla v_m \rightharpoonup  \nabla v_\infty \text{ weakly in } L^2([0,T]\times \R^d),\\
        v_m \rightharpoonup v_\infty \text{ weakly in } L^\infty(0,T;L^2(\R^d)).
    \end{cases}
\end{equation*}
\begin{lemma}
    The limit functions $\rho_\infty,$ $n_\infty$ satisfy 
    $$
    \rho_\infty(t,x)=\int_{0}^{\infty}V(\theta)n_\infty(t,\theta,x) d\theta
    $$ almost everywhere.
\end{lemma}
\begin{proof}
    This follows from the fact that $\rho_m(t,x)=\int_{0}^{\infty}V(\theta)n_m(t,\theta,x)d\theta,$ the weak-* convergences of $\rho_m$ and $n_m$, and Lemma \ref{t.cptspt}.
\end{proof}
Remark that \eqref{rho.comp}, the equation on $\rho_m,$ can be re-written as \begin{equation}\label{v.eq}
\partial_t \rho_m -\Delta v_m=\int_{0}^{\infty}n_mF(\theta,p_m)d\theta.
\end{equation}
Using the convergences listed above, we are able to show that the left-hand side of equation \eqref{rho.comp} can be passed to the limit. However, to pass to the limit on the right-hand side, strong convergence in $p_m$ is required (because of the non-linearity $F$). As such, we are only able to show so far that the right hand side converges to some $L^\infty$ function, $\mathcal{H}_\infty.$ This is outlined in the following lemma.
\begin{lemma}[\cite{david}, Corollary 4.5]\label{weaklemma}
    The limit functions satisfy $0 \leq \rho_\infty \leq 1$, $0 \leq p_\infty \leq p_M$, $v_\infty=p_\infty,$ and $p_\infty(1-\rho_\infty)=0$ almost everywhere in $[0,T]\times \R^d.$ The limit functions, $p_\infty, \rho_\infty$ satisfy $p_\infty \in L^2(0,T;H^1(\R^d)),$ $\partial_t \rho_\infty \in L^2(0,T;H^{-1}(\R^d)),$ and solve
    \begin{equation}\label{lim.eq}
        \partial_t \rho_\infty -\Delta p_\infty =\mathcal{H}_\infty 
    \end{equation}
    in $\mathcal{D}'([0,T]\times\R^d)$,
    for some $\mathcal{H}_\infty(t,x) \in L^\infty([0,T]\times \R^d)$, where $\mathcal{H}_\infty$ is the weak-* limit in $L^\infty$ of $\int_{0}^{\infty}n_m(F(\theta,p_m)-\mu(\theta)V(\theta))d\theta.$
\end{lemma}
\begin{proof}
    The bounds on $\rho_\infty$ and $p_\infty$ follow from the $L^\infty$ bounds on $\rho_m$ and $p_m$. 
    
    To show that $p_{\infty}=p_{\infty}\rho_{\infty}$ almost everywhere, consider $v_m=\rho_m^m=\frac{m-1}{m}\rho_mp_m.$ From Lemma \ref{q.lemma}, we have that $\rho_m, p_m \in L^\infty(0,T;L^2(\R^d))$ with weak $L^2$ limits $\rho_\infty, p_\infty$. Further,  from Lemma \ref{secondbds}, $\partial_t \rho_m \in L^2(0,T;H^{-1}(\R^d)),$ and $p_m \in L^2(0,T;H^1(\R^d)).$ By the compensated compactness theorem (\cite{david}, Appendix A), up to a subsequence,
    \begin{equation*}
       \lim_{m\rightarrow \infty} \int_{0}^{T} \int_{\R^d} v_m \varphi \, dx dt=\lim_{m\rightarrow \infty}\int_{0}^{T}\int_{\R^d}\frac{m-1}{m}\rho_m p_m \varphi\,dx dt = \int_0^T \int_{\R^d} \rho_\infty p_\infty \varphi \, dx dt,
    \end{equation*}
    for all $\varphi \in C(0,T;C^1(\R^d)).$ But since $v_m \rightharpoonup v_\infty$ in $L^\infty(0,T;L^2(\R^d)),$ by uniqueness of the weak limit, $v_\infty =\rho_\infty p_\infty$ almost everywhere.

    Further, since $x \mapsto x^{\frac{m}{m-1}}$ is convex, by weak lower-semicontinuity, we can show that
    $$
    v_\infty  =\liminf_{m \rightarrow \infty}(\frac{m-1}{m}p_m)^{\frac{m}{m-1}}\geq p_\infty.$$
    To see this, define $\Psi_m(x)=\frac{m-1}{m}x^{\frac{m}{m-1}},$ and define $\psi_\delta(x)=\max \{x-\delta,0\}$. Notice that $\psi_\delta$ is convex and $\psi_\delta(x)\leq\Psi_m(x)$ for large $m$. Then, by lower-semicontinuity of $\psi_\delta$,
    $$
    \psi_\delta(p_\infty)\leq \liminf_{m\rightarrow\infty}\psi_\delta(p_m)\leq \liminf_{m \rightarrow\infty}\Psi_m(p_m)=\liminf_{m\rightarrow\infty}v_m=v_\infty.
    $$
    Taking $\delta \rightarrow 0$, we have $
    p_\infty \leq v_\infty.$
    At the same time, because $\rho_\infty \leq 1$, it holds that
    $$
    v_\infty =\rho_\infty p_\infty \leq p_\infty,$$
meaning that $p_\infty=v_\infty.$
  From here, we conclude that $p_\infty=\rho_\infty p_\infty$ a.e. In particular, this proves that $p_\infty$ satisfies \eqref{hs}.

    To show \eqref{lim.eq}, using the weak-* $L^\infty([0,T]\times\R^d)$ convergence of $\rho_m$ and the weak $L^2([0,T]\times\R^d)$ convergence of $\nabla v_m$, the left hand side of \eqref{v.eq} can be passed to the weak-* limit. On the right hand side, recall that $$\int_{0}^{\infty}n_mF(\theta,p_m)\, d\theta\leq C\rho_m,$$ so this term is bounded in $L^\infty([0,T]\times \R^d)$, and converges weakly-* to some limit $\mathcal{H}_\infty$ in $L^\infty([0,T]\times\R^d).$
\end{proof}
What remains is to identify the function $\mathcal{H}_\infty$, and pass to the limit in the equation on $n_m,$ \eqref{nm}. In  order to do so, strong convergence is required in both $p_m$ (to pass to the limit in the reaction term to identify $\mathcal{H}_\infty$) and $\nabla p_m$ (to pass to the limit in div$(n_m \nabla p_m)$). It is easier to first show that $\nabla v_m$ converges strongly to $\nabla v_\infty$, and because $v_\infty=p_\infty$, the strong convergence of $p_m$ and $\nabla p_m$ will follow easily.

The proof of the following lemma follows Lemma 4.6 in \cite{david}. Some parts of the proof are identical, but are included for completion. The main difference  in our case appears due to the transport in $\theta$, which means that it cannot be treated as a parameter. Because of this, the compensated compactness theorem cannot be directly applied, as it could in \cite{david}.
\begin{lemma}[\cite{david}, Lemma 4.6]\label{lem.strong_grad}
    Up to a subsequence, as $m \rightarrow \infty$, $\nabla v_m \rightarrow \nabla v_\infty$ strongly in $L^2((0,T)\times \R^d).$
\end{lemma}
\begin{proof}
To show the strong convergence of $\nabla v_m,$
      \begin{align*}
        \limsup_{m  \rightarrow \infty}\int_{0}^{T}\int_{\R^d}|\nabla (v_m-v_\infty)|^2\, dx dt &=\limsup_{m \rightarrow \infty}\int_{0}^{T}\int_{\R^d}\nabla v_m\cdot \nabla (v_m-v_\infty)-\nabla v_\infty\cdot \nabla(v_m-v_\infty)\, dx dt  \nonumber \\&\leq \limsup_{m \rightarrow \infty}\int_{0}^{T}\int_{\R^d}\nabla v_m\cdot \nabla (v_m-v_\infty)\, dx dt,
    \end{align*}
    because $\nabla (v_m-v_\infty)\rightharpoonup 0$ weakly in $L^2([0,T]\times \R^d).$
    Testing \eqref{v.eq} against $v_m-v_\infty$, 
    \begin{align*}
       \int_{\R^d} \partial_t \rho_m(v_m-v_\infty)\, dx +\int_{\R^d}\nabla v_m\cdot \nabla(v_m-v_\infty)\, dx=\int_{\R^d}(v_m-v_\infty)\int_{0}^{\infty}n_mF(\theta,p_m)\, d\theta dx.
    \end{align*}
   
    Rearranging,
    \begin{align*}
        \limsup_{m \rightarrow \infty}\int_{0}^{T}\int_{\R^d}\nabla v_m\cdot \nabla (v_m-v_\infty)\, dx dt&=\limsup_{m \rightarrow \infty}\int_{0}^{T}\int_{\R^d}(v_m-v_\infty)\int_{0}^{\infty}F(\theta,p_m)n_m\, d\theta dx\\&+\limsup_{m\rightarrow\infty}(-\int_{0}^{T}\int_{\R^d}\partial_t \rho_m(v_m-v_\infty)\, dx dt).
    \end{align*}
    Notice that
    \begin{align*}
        \partial_t \rho_m v_m=\frac{1}{m+1}\partial_t(\rho_m^{m+1}),
    \end{align*}
    so 
    \begin{align*}
    \limsup_{m\rightarrow\infty}(- \int_{0}^{T}\int_{\R^d}\partial_t \rho_m v_m\, dx dt) &=\limsup_{m\rightarrow\infty}(-\frac{d}{dt}\int_{0}^{T}\int_{\R^d}\frac{1}{m+1}\rho_m^{m+1}\,dxdt) \\
        & =\limsup_{m\rightarrow\infty}\frac{1}{m+1}\int_{\R^d}\rho_m^{m+1}(0)-\rho_m^{m+1}(T)\, dx dt \\
        & = 0,
    \end{align*}
     since $\rho_m^{m+1}(T)$ and $\rho_m^{m+1}(0)$ are bounded in $L^1(\R^d)$ uniformly in $m$. Then,
    \begin{align}
        \limsup_{m \rightarrow \infty}\int_{0}^{T}\int_{\R^d}\nabla v_m\cdot \nabla (v_m-v_\infty)\, dx dt&=\limsup_{m \rightarrow \infty}\int_{0}^{T}\int_{\R^d}(v_m-v_\infty)\int_{0}^{\infty}F(\theta,p_m)n_m\, d\theta dx \nonumber \\&+\limsup_{m \rightarrow \infty}\int_{0}^{T}\int_{\R^d}\partial_t \rho_m v_\infty\, dx dt \nonumber \\
        &= \limsup_{m \rightarrow \infty}\int_{0}^{T}\int_{\R^d}(v_m-v_\infty)\int_{0}^{\infty}(F(\theta,p_m)-F(\theta,p_\infty))n_m\, d\theta dx dt \label{t.1}\\&+\limsup_{m \rightarrow \infty}\int_{0}^{T}\int_{\R^d}(v_m-v_\infty)\int_{0}^{\infty}F(\theta,p_\infty)n_m\,d\theta dx dt \label{t.2}\\
    &+\int_{0}^{T}\int_{\R^d}\partial_t \rho_\infty v_\infty\, dx dt, \label{t.3}
    \end{align}
    using the fact that of $\partial_t \rho_m $ is bounded in $L^2(0,T;H^{-1}(\R^d))$ to get \eqref{t.3}.

  To conclude the lemma, it remains to show that all three terms \eqref{t.1}, \eqref{t.2}, \eqref{t.3} are zero. The proofs of \eqref{t.1} and \eqref{t.3} follow closely the proof in \cite{david}. The proof of \eqref{t.2} is the main difference between this proof and the one in \cite{david}, and requires a generalized version of the compensated compactness theorem.
  
\medskip
 \noindent \textbf{Proof that \eqref{t.1} equals zero.} We have:
   
   \begin{align}\label{split}
       &\int_{0}^{T}\int_{\R^d}\int_{0}^{\infty}(F(\theta,p_m)-F(\theta,p_\infty))n_m(v_m-v_\infty)\,d\theta dx dt  \nonumber\\&=\int_{0}^{T}\int_{\R^d}\int_{0}^{\infty}(F(\theta,p_m)-F(\theta,p_\infty))n_m(p_m(\frac{m-1}{m}\rho_m-1)+p_m-p_\infty)\,d\theta dx dt \nonumber \\
       &\leq \int_{0}^{T}\int_{\R^d}\int_{0}^{\infty}(F(\theta,p_m)-F(\theta,p_\infty))n_m(p_m(\frac{m-1}{m}\rho_m-1))\, d\theta dx dt.
   \end{align}
  In the above calculation, the first line comes from the definition of $v_m$, and the fact that $p_\infty=v_\infty$. The second line comes from the fact that $F(\theta,p)$ is nonincreasing in $p$ so $(F(\theta,p_m)-F(\theta,p_\infty))(p_m-p_\infty)\leq 0.$
  Then, for a fixed $\varepsilon>0,$
  \begin{align*}
    \int_{0}^{T}\int_{\R^d}\int_{0}^{\infty}&(F(\theta,p_m)-F(\theta,p_\infty))n_m(p_m(\frac{m-1}{m}\rho_m-1))\, d\theta dx dt\\&=\iint_{\rho_m\leq 1-\varepsilon}\int_{0}^{\infty}(F(\theta,p_m)-F(\theta,p_\infty))n_m\,d\theta (\rho_m^{m-1}(\rho_m-\frac{m}{m-1}))\,dxdt \\&+\iint_{\rho_m> 1-\varepsilon}\int_{0}^{\infty}(F(\theta,p_m)-F(\theta,p_\infty))n_m\,d\theta (p_m(\frac{m-1}{m}\rho_m-1))\,dxdt. 
  \end{align*}
  For the first term,

\begin{align*}
    \iint_{\rho_m\leq 1-\varepsilon}\int_{0}^{\infty} &(F(\theta,p_m)-F(\theta,p_\infty))n_m\,d\theta (\rho_m^{m-1}(\rho_m-\frac{m}{m-1}))\,dxdt \\
    & \leq C\rho_M(1-\varepsilon)^{m-1}\iint_{\rho_m\leq 1-\varepsilon}\rho_m +\frac{m}{m-1}\,dxdt \\
    & \leq C(T)(1-\varepsilon)^{(m-1)},
\end{align*}
using the bound on $F$ in the first inequality, and the $L^1((0,T)\times\R^d)$ bound on $\rho_m$, and the compact support of $\rho_m$ in the second inequality.
Also,

  \begin{align*}
      \iint_{\rho_m> 1-\varepsilon}\int_{0}^{\infty}&(F(\theta,p_m)-F(\theta,p_\infty)) n_m\,d\theta (p_m(\frac{m-1}{m}\rho_m-1))\,dxdt \\& \leq C\rho_M\iint_{\rho_m>1-\varepsilon}p_m|\frac{m-1}{m}\rho_m-1|\,dxdt \\
      & \leq C(T)\sup_{\rho_m>1-\varepsilon}|\frac{m-1}{m}\rho_m-1|,
  \end{align*}
  using the $L^1([0,T]\times \R^d)$ bound on $p_m$.
 To bound $\sup_{\rho_m>1-\varepsilon}|\frac{m-1}{m}\rho_m-1|$, first notice that $\rho_m\leq e^{\frac{1}{m-1}\ln(p_M)}\leq 1+\frac{1}{m-1}\ln(p_M)+o(\frac{1}{m-1})$ and so $\frac{m-1}{m}\rho_m-1\leq o(\frac{1}{m-1}).$ 
 
 Also, when $\rho_m-1>-\varepsilon,$ $\frac{m-1}{m}\rho_m-1> -\varepsilon-\frac{1}{m}\rho_M.$ Combining these bounds, for $\rho_m>1-\varepsilon,$ it holds that
  
  $$
  |\frac{m-1}{m}\rho_m-1|\leq \max(\frac{1}{m}\rho_M+\varepsilon,o(\frac{1}{m-1})).$$ 

  Taking $\varepsilon=\frac{1}{\sqrt{m-1}},$

  \begin{align*}
      \limsup_{m\rightarrow\infty}\int_{0}^{T}\int_{\R^d}&\int_{0}^{\infty}(F(\theta,p_m)-F(\theta,p_\infty))n_m(p_m(\frac{m-1}{m}\rho_m-1))\, d\theta dx dt\\& \leq \limsup_{m \rightarrow \infty}C(T)(1-\frac{1}{\sqrt{m-1}})^{m-1}+C(T)\max(\frac{1}{m}\rho_M+\frac{1}{\sqrt{m-1}}, o(\frac{1}{m-1})) = 0.
  \end{align*}

  \medskip
  \noindent \textbf{Proof that \eqref{t.2} equals zero.} This will be done using an extended version of the proof of the Compensated Compactness Theorem found in Appendix A of \cite{david}, and is the main difference between this Lemma and Lemma 4.6 in \cite{david}.

 For simplicity of notation, define $w_m(t,\theta,x)=(v_m-v_\infty)F(\theta,p_\infty)$.
 Recall that $w_m \in L^\infty([0,T]\times \R^d; W^{1,\infty}([0,\infty)))\cap L^\infty(0,\infty;L^2(0,T;H^1(\R^d))) $ and $w_m \rightharpoonup 0$ weakly in that space. In what follows, we will use mollification to obtain a strong convergence, and show that the product $w_m n_m$ converges.
 We have 
 $$
 \int_{0}^{T}\int_{\R^d}\int_{0}^{\infty}(v_m-v_\infty)F(\theta,p_\infty)n_md\theta dx dt =\int_{0}^{T}\int_{\R^d}\int_{0}^{\infty}w_m n_m d\theta dxdt.
 $$
 Take $\psi_\varepsilon(x)=\frac{1}{\varepsilon^d}\psi(\frac{x}{\varepsilon}),$ and $\zeta_\sigma (t)=\frac{1}{\sigma}\zeta(\frac{t}{\sigma}),$ where $\psi$ and $\zeta$ are standard mollifiers.

 Then,
 \begin{align*}
     \int_{0}^{T}\int_{\R^d}\int_{0}^{\infty}w_mn_md\theta dx dt &=\underbrace{\int_{0}^{T}\int_{\R^d}\int_{0}^{\infty}n_m(w_m -w_m*_x \psi_\varepsilon)d\theta dx dt}_{\mathcal{A}_m}+\\&\underbrace{\int_{0}^{T}\int_{\R^d}\int_{0}^{\infty}(n_m-n_m *_t \zeta_\sigma)(w_m *_x \psi_\varepsilon) d\theta dx dt }_{\mathcal{B}_m}
     \\&+\underbrace{\int_{0}^{T}\int_{\R^d}\int_{0}^{\infty}(n_m *_t \zeta_\sigma)(w_m *_x\psi_\varepsilon)d\theta dx dt}_{\mathcal{C}_m}.
 \end{align*}
 We will show that 
 $$
 \limsup_{m \rightarrow \infty}|\int_{0}^{T}\int_{\R^d}\int_{0}^{\infty}w_mn_md\theta dx dt|\leq C(\varepsilon,\sigma),
 $$
 where $\lim_{\varepsilon,\sigma \rightarrow 0}C(\varepsilon,\sigma)=0,$
 proving that \eqref{t.2} is zero. 
 
 First, consider the term $\mathcal{C}_m:$ 
 \begin{align*}
    |\mathcal{C}_m|&= |\int_{0}^{T}\int_{\R^d}\int_{0}^{\infty}(n_m *_t \zeta_\sigma)(w_m *_x\psi_\varepsilon)d\theta dx dt|\\&=|\int_{0}^{T}\int_{\R^d}\int_{0}^{\infty}n_m w_m *_x \psi_\varepsilon *_t \zeta_\sigma d\theta dx dt| \\
     & \leq ||w_m *_x \psi_\varepsilon *_t \zeta_\sigma ||_{L^\infty([0,T]\times[0,\infty)\times\R^d)}||n_m||_{L^\infty([0,T]\times\R^d;L^1([0,\infty)))}.
 \end{align*}
 
 Thanks to the mollification, for any $\varepsilon, \sigma>0,$ $w_m *_x \psi_\varepsilon*_t \zeta_\sigma$ has bounded derivatives in $x$ and $t$, so,  

 $$
 w_m *_x \psi_\varepsilon*_t \zeta_\sigma \in W^{1,\infty}([0,T]\times[0,\infty)\times\R^d).
 $$

By Arzela-Ascoli, and using the fact that $w_m$ is compactly supported in all three variables (from Lemmas \ref{lem.cptspt}, \ref{t.cptspt}), $w_m *_x \psi_\varepsilon*_t \zeta_\sigma$ converges  uniformly up to a subsequence. Using the fact that $w_m \rightharpoonup 0$ weakly, it must be that $w_m *_x \psi_\varepsilon *_t \zeta_\sigma \rightarrow 0$ uniformly.   

Using the fact that $n_m$ is bounded in $L^\infty([0,T]\times \R^d ;L^1([0,\infty))$ uniformly in $m$, it follows that
$$
\limsup_{m\rightarrow \infty}|\mathcal{C}_m|=0.
$$

Next, consider the term $\mathcal{A}_m:$
\begin{align*}
|\mathcal{A}_m|&=|\int_{0}^{T}\int_{\R^d}\int_{0}^{\infty}n_m(w_m -w_m*_x \psi_\varepsilon)d\theta dx dt|\\&=|\int_{0}^{T}\int_{\R^d}\int_{0}^{\infty}\Bigl(\int_{\R^d}(w_m(t,\theta,x)-w_m(t,\theta,x-\varepsilon y))\psi(y)dy\Bigr)n_m(t,\theta,x)d\theta dx dt |\\&= |\int_{0}^{T}\int_{\R^d}\int_{0}^{\infty}\int_{\R^d}\int_{x-\varepsilon y}^{x}\nabla_x w_m(t,\theta,z)dz\psi(y)n_m(t,\theta,x) dy d\theta dx dt| \\
& \leq \int_{0}^{T}\int_{\R^d}\int_{0}^{\infty}\int_{\R^d}||\nabla_x w_m(t,\theta,\cdot)||_{L^2(\R^d)}(\varepsilon |y|)^{\frac{1}{2}}n_m(t,\theta,x)dyd\theta dx dt\\
& \leq \int_{\R^d}||\nabla_x w_m||_{L^\infty(0,\infty; L^2([0,T]\times \R^d))}(\varepsilon |y|)^{\frac{1}{2}}\psi(y)||n_m||_{L^\infty(0,T;L^1([0,\infty)\times \R^d))}dy \\
&\leq C\int_{\R^d}\sqrt{(\varepsilon|y|)}\psi(y)dy \\
& \leq C\sqrt{\varepsilon},
\end{align*}
 using the fact that $w_m \in L^\infty(0,\infty;L^2(0,T;H^1(\R^d)),$ and $n_m \in L^\infty(0,T;L^1([0,\infty)\times \R^d),$ both uniformly in $m$. So, 
 $$\limsup_{m\rightarrow \infty}|\mathcal{A}_m|\leq C\sqrt{\varepsilon},$$
 for a constant $C$ not depending on $m$.
Lastly, consider the term $\mathcal{B}_m:$
\begin{align*}
|\mathcal{B}_m|&=|\int_{0}^{T}\int_{\R^d}\int_{0}^{\infty}(n_m-n_m *_t \zeta_\sigma)(w_m *_x \psi_\varepsilon) d\theta dx dt| \\&= |\int_{0}^{T}\int_{\R^d}\int_{0}^{\infty}\Bigl(\int_{0}^{\infty}(n_m(t)-n_m(t-\sigma \tau) )\zeta(\tau)d\tau\Bigr)(w_m *_x \psi_\varepsilon) d\theta dx dt|\\
 &= |\int_{0}^{T}\int_{\R^d}\int_{0}^{\infty}\Bigl(\int_{0}^{\infty}\int_{t-\sigma \tau}^{t}\partial_t n_m(s)ds\zeta(\tau)d\tau\Bigr)(w_m *_x \psi_\varepsilon) d\theta dx dt| \\
 &= |\int_{0}^{T}\int_{\R^d}\int_{0}^{\infty}\int_{0}^{\infty}\int_{t-\sigma \tau}^{t}(\underbrace{-r(p_m)\partial_\theta n_m (s)}_{i}+\underbrace{\div(n_m \nabla p_m(s))}_{ii}\\&\underbrace{-r(p_m)\nu(\theta,p_m)n_m(s)-\mu(\theta)n_m(s)}_{iii})ds \zeta(\tau)d\tau (w_m*_x \psi_\varepsilon)d\theta dx dt|.
\end{align*}
For $i,$ taking the $\theta$ integral and integrating by parts,
\begin{align*}
    |&\int_{0}^{T}\int_{\R^d}\int_{0}^{\infty}\int_{0}^{\infty}\int_{t-\sigma \tau}^{t}-r(p_m(s,x))\partial_\theta n_m(s,\theta,x)\zeta(\tau)(w_m *_x \psi_\varepsilon)(t,\theta,x)dsd\tau d\theta dx dt| \\
    &=|\int_{0}^{T}\int_{\R^d}\int_{0}^{\infty}\int_{t-\sigma \tau}^{t}[(w_m*_x \psi_\varepsilon)(t,0,x)2r(p_m(s,x))\int_{0}^{\infty}\nu(\theta,p_m)n_m(s,\theta,x)d\theta\\&+r(p_m(s,x))\int_{0}^{\infty}n_m(s,\theta,x)\partial_\theta (w_m *_x \psi_\varepsilon)(t,\theta,x)d\theta] \zeta(\tau) ds d\tau dx dt |\\
    &\leq \Bigl(||C(w_m *_x \psi_\varepsilon) n_m||_{L^\infty([0,T]\times \R^d;L^1([0,\infty)))}\\&+||Cn_m \partial_\theta (w_m *_x \psi_\varepsilon)||_{L^\infty([0,T]\times \R^d; L^1([0,\infty)))}\Bigr)\int_{0}^{T}\int_{0}^{\infty}\sigma \tau \zeta(\tau)d\tau dt \\
    & \leq C(T)\int_{0}^{\infty}\sigma \tau \zeta(\tau) d\tau  \\
    & \leq C\sigma,
\end{align*}
using that $\partial_\theta w_m, w_m
\in L^\infty([0,T]\times[0,\infty)\times\R^d)$, and $n_m \in L^\infty([0,T]\times \R^d;L^1([0,\infty)))$, bounded uniformly in $m$, and the compact support of solutions.
For $ii$, taking the integral in $x$,
\begin{align*}
    |&\int_{0}^{T}\int_{\R^d}\int_{0}^{\infty}\int_{0}^{\infty}\int_{t-\sigma \tau}^{t}\div(n_m\nabla p_m(s))ds \zeta(\tau) d\tau (w_m *_x \psi_\varepsilon)d\theta dx dt |\\
    &= |-\int_{0}^{T}\int_{\R^d}\int_{0}^{\infty}\int_{0}^{\infty}\int_{t-\sigma \tau}^{t}n_m(s) \nabla p_m(s)\nabla (w_m *_x \psi_\varepsilon)(t)\zeta(\tau) ds d\tau d\theta dx dt |\\
    & \leq \int_{0}^{\infty}\int_{t-\sigma \tau}^{t} ||\rho_m(s)||_{L^\infty(\R^d)}||\nabla p_m (s)||_{L^2(\R^d)}||\nabla (w_m *_x \psi_\varepsilon)(t)||_{L^\infty(0,\infty;L^2(\R^d))} \zeta(\tau)d\tau ds dt\\
    & \leq C(T)\int_{0}^{\infty}||\rho_m||_{L^\infty([0,T]\times\R^d}||\nabla p_m||_{L^2([0,T]\times \R^d)}||\nabla (w_m *_x \psi_\varepsilon )||_{L^\infty(0,\infty;L^2([0,T]\times \R^d))}\zeta(\tau)\sqrt{\sigma \tau} d\tau \\
    & \leq C\int_{0}^{\infty}\sqrt{\sigma \tau} \zeta(\tau)d\tau \\
    & \leq C\sqrt{\sigma},
\end{align*}
using that $\nabla p_m \in L^2([0,T]\times \R^d), $ $\nabla (w_m *_x \psi_\varepsilon) \in L^\infty(0,\infty;L^2([0,T]\times\R^d))$ and $n_m \in L^\infty([0,T]\times 
\R^d;L^1([0,\infty))),$ bounded uniformly in $m$.
For $iii$,
\begin{align*}
    &|\int_{0}^{T}\int_{\R^d}\int_{0}^{\infty}\int_{0}^{\infty}\int_{t-\sigma \tau}^{t}(-r(p_m)\nu(\theta,p_m)-\mu(\theta))n_m(s)\zeta(\tau)(w_m *_x \psi_\varepsilon)dsd\tau d\theta dx dt| \\
    &\leq C\int_{0}^{\infty}\tau \sigma \zeta(\tau)d\tau \\
    & \leq C\sigma,
\end{align*}
using that $(w_m *_x \psi_\varepsilon)\in L^\infty([0,T]\times[0,\infty)\times \R^d)$ and $(r(p_m)\nu(\theta,p_m)-\mu(\theta))n_m \in L^\infty([0,T]\times \R^d;L^1([0,\infty))),$ bounded uniformly in $m$, and the compact support of solutions. So,
$$
\limsup_{m\rightarrow \infty}|\mathcal{B}_m|\leq C\sigma,
$$
for a constant $C$ independent of $m$.

Combining $\mathcal{A}_m,\mathcal{B}_m,$ and $\mathcal{C}_m$, we have 
$$
\limsup_{m\rightarrow \infty}|\int_{0}^{T}\int_{\R^d}\int_{0}^{\infty}w_mn_m d\theta dx dt| \leq \limsup_{m\rightarrow \infty}|\mathcal{A}_m|+|\mathcal{B}_m|+|\mathcal{C}_m|\leq C\sqrt{\varepsilon}+C\sigma,
$$
for constants $C$ not depending on $m$. Taking $\varepsilon, \sigma \rightarrow 0,$
we have that \eqref{t.2} is equal to zero.

  \medskip
  \noindent \textbf{Proof that \eqref{t.3} equals zero.} 
  Formally, this should hold because $v_\infty >0$ only when $\rho_\infty =1,$ i.e., $\partial_t \rho_\infty =0.$
However, because $\partial_t \rho_\infty $ is only in $ L^2(0,T;H^{-1}(\R^d))$, this needs to be justified. 
 This proof follows closely from \cite{david}, where further details can be found. First, notice that
  \begin{align*}
     \frac{1}{\varepsilon} \int_0^T\int_{\R^d}(\rho_{\infty}(t+\varepsilon,x)-\rho_{\infty}(t,x))v_{\infty}\,dxdt&=\frac{1}{\varepsilon}\int_{0}^{T}\int_{\R^d}(\rho_{\infty}(t+\varepsilon,x)-1)v_\infty+(1-\rho_{\infty}(t,x))v_{\infty}\,dxdt  
      \leq 0,
  \end{align*}
  because $(1-\rho_\infty)v_\infty=0$ and  $\rho_\infty \leq 1$. Similarly,
  and
  \begin{align*}
     \frac{1}{\varepsilon} \int_\varepsilon^T\int_{\R^d}(\rho_{\infty}(t,x)-\rho_{\infty}(t-\varepsilon,x))v_{\infty}\,dxdt=\frac{1}{\varepsilon}\int_{\varepsilon}^{T}\int_{\R^d}(\rho_{\infty}(t,x)-1)v_\infty+(1-\rho_{\infty}(t-\varepsilon,x))v_{\infty}\,dxdt
      &\geq 0.
  \end{align*}
  
  Recall from Lemma \ref{weaklemma} that $v_\infty \in L^2(0,T;H^1(\R^d))$ and $\partial_t \rho_\infty \in L^2(0,T;H^{-1}(\R^d)).$ Then, 
  \begin{align*}
 0&= \lim_{\varepsilon \rightarrow 0} \int_{0}^{T}\int_{\R^d}(\rho_\infty(t+\varepsilon,x)-\rho_\infty(t,x))v_\infty dx dt \\&=\lim_{\varepsilon \rightarrow 0}\int_{0}^T \langle\frac{1}{\varepsilon}(
  \rho_\infty(t+\varepsilon,x)-\rho_\infty(t,x),v_\infty \rangle_{H^{-1}(\R^d),H^1(\R^d)}dt\\&=\int_{0}^{T}\langle \partial_t \rho_\infty (t,x),v_\infty \rangle_{H^{-1}(\R^d),H^1(\R^d)} dt \\
  &=\int_{0}^{T}\int_{\R^d}\partial_t \rho_\infty v_\infty dx dt.
  \end{align*}
 
  Because \eqref{t.1}, \eqref{t.2} and \eqref{t.3} are all $0$, the proof of the lemma is complete.
 
\end{proof}

\begin{lemma}[\cite{david}, Corollary 4.7]
    Up to a subsequence, $p_m \rightarrow p_\infty$ strongly in $L^2([0,T]\times \R^d).$
\end{lemma}
\begin{proof}
    From Lemma \ref{lem.strong_grad}, and Lemma \ref{lem.cptspt}, and Poincare's inequality, $v_m \rightarrow v_\infty$ strongly in $L^2([0,T]\times \R^d).$ Then since $p_m=\frac{m-1}{m}v_m^\frac{m-1}{m},$ $p_m \rightarrow v_\infty=p_\infty$ as $m \rightarrow \infty$, strongly in $L^2([0,T]\times \R^d)$.
\end{proof}
\begin{lemma}
     Up to a subsequence, $\nabla p_m \rightarrow \nabla p_\infty$ strongly in $L^2([0,T]\times \R^d).$
\end{lemma}
\begin{proof}
Assume $m>4$. First, $$|\nabla p_m|^2=\nabla v_m \cdot \nabla(\frac{m}{m-2}\rho_m^{m-2}).$$ Because $\nabla v_m\rightarrow \nabla v_\infty$ strongly in $L^2([0,T]\times\R^d),$ it suffices to show that $\nabla (\frac{m}{m-2}\rho_m^{m-2})$ is bounded in $L^2([0,T]\times \R^d).$ Once this is shown, it will hold that $||\nabla p_m||_{L^2} \rightarrow ||\nabla p_\infty||_{L^2}.$ To see this, notice that $\frac{m}{m-2}\rho_m^{m-2}=\frac{m}{m-2}v_m^{\frac{m-2}{m}} \rightarrow p_\infty$, because $v_m$ converges strongly to $v_\infty=p_\infty.$ Since it already is known that $\nabla p_m \rightharpoonup \nabla p_\infty$ weakly in $L^2([0,T]\times \R^d),$ the weak convergence combined with the convergence in norm implies that $\nabla p_m \rightarrow \nabla p_\infty$ strongly in $L^2([0,T]\times \R^d).$ 

To show the desired $L^2$ bound, $$
|\frac{m}{m-2}\nabla \rho_m^{m-2}|^2=m^2\rho_m^{2m-6}|\nabla \rho_m|^2.$$ Then, using \eqref{q.norm} with $q=m-3,$ $$
\frac{d}{dt}\int_{\R^d}\rho_m^{m-3}\,dx +m(m-3)(m-4)\int_{\R^d}\rho_m^{2m-6}|\nabla \rho_m|^2\, dx \leq (m-3)\int_{\R^d}\rho_m^{m-3}\,dx.$$ Multiplying by $\frac{m}{m-2},$ $$
\frac{d}{dt}\int_{\R^d}\frac{m}{m-2}\rho_m^{m-3}
dx+\frac{m-3}{m-2}(m-4)\int_{\R^d}|\nabla (\frac{m}{m-2}\rho_m^{m-2})|^2\,dx\leq (m-3)\int_{\R^d}\frac{m}{m-2}\rho_m^{m-3}\,dx.$$ Integrating in time,
\begin{equation*}
    \int_{0}^{T}\int_{\R^d}|\nabla (\frac{m}{m-2}\rho_m^{m-2})|^2dxdt \leq \frac{m}{m-4}\int_0^T\int_{\R^d}\rho_m^{m-3}dxdt+\frac{m}{(m-3)(m-4)}\int_{\R^d}\rho_m^{m-3}(0,x)\,dx.
\end{equation*}
Then, since $\rho_m^{m-3}=(\frac{m-1}{m}p_m)^{\frac{m-3}{m-1}} \leq \max\{p_M,1\}$, and the compact support of solutions is uniform in $m$, the right hand side is bounded independent of $m$.
\end{proof}
Finally, we present a lemma that demonstrates that $n_m \log n_m$ is bounded in $L^1([0,\infty))$ in $\theta$, and by Dunford-Pettis theorem, this will imply that $n_\infty$ is not a measure, and the sequence $n_m$ converges in $L^1([0,\infty))$ in $\theta$. This lemma is similar to Lemma 4.5 in \cite{llmw}.
\begin{lemma}\label{lem.nlogn}
    The function $n_m$ satisfies the following:
    \begin{equation}\label{nlogn}
        \sup_{t \in [0,T]}\int_{\R^d}\int_{0}^{\infty}n_m(t,\theta,x)\log_{+}n_m(t,\theta,x)\,d\theta dx \leq C(T),
    \end{equation}
    for a constant $C(T)$ independent of $m$, only depending on $T$ and the uniform in $m$ bound on the initial condition.
\end{lemma}
\begin{proof}
    Begin by multiplying \eqref{nm} by $(1+\log n_m)V(\theta)$ and integrating in $\theta$ and $x$ to get the following equation:
    \begin{multline*}
    \int_{\R^d}\int_{0}^{\infty}\partial_t (n_m \log n_m)V(\theta)+\int_{\R^d}\int_{0}^{\infty}r(p_m)\partial_\theta(n_m \log n_m)V(\theta)d\theta dx\\=
    \int_{\R^d}\int_{0}^{\infty}-r(p_m)\nu(\theta,p_m)n_m(\log n_m+1)V(\theta)-\mu(\theta)n_m(\log n_m+1)V(\theta)d\theta dx -\int_{\R^d}\nabla \rho_m \cdot \nabla p_m dx.
    \end{multline*}
    Then, integrating $\partial_\theta (n_m \log n_m)V(\theta)$ in $\theta$, recalling that $n_m \geq 0$, and noticing that $\nabla \rho_m \cdot \nabla p_m =\frac{m^2}{4}\rho_m^{m-1}|\nabla \rho_m|^2 \geq 0,$ rearranging, we have,
    \begin{multline*}
        \int_{\R^d}\int_{0}^{\infty}\partial_t (n_m \log n_m)V(\theta)d\theta dx \leq \int_{\R^d}r(p_m)V(0)n_m(t,0,x)\log n_m (t,0,x)dx\\+\int_{\R^d}\int_{0}^{\infty}n_m \log n_m (V'(\theta)-r(p_m)\nu(\theta,p_m)V(\theta)-\mu(\theta)V(\theta))d\theta dx \\
         \leq \int_{\R^d}r(p_m)V(0)n_m(t,0,x)\log_+ n_m (t,0,x)dx\\+\int_{\R^d}\int_{0}^{\infty}n_m \log n_m (V'(\theta)-r(p_m)\nu(\theta,p_m)V(\theta)-\mu(\theta)V(\theta))d\theta dx .
    \end{multline*}
   Then, because $n_m(t,0,x)\leq C\rho_m$, for a constant $C$ depending only on $\nu$ and $V$, from the monotonicity of $s \mapsto s\log_+ s,$
   \begin{align*}
   \int_{\R^d}\int_{0}^{\infty}\partial_t(n_m \log n_m)V(\theta)d\theta dx &\leq C\int_{\R^d} \rho_m \log_+ \rho_mdx +C\int_{\R^d}\int_{0}^{\infty}|n_m \log n_m|V(\theta) \\&\leq C(T)+C\int_{\R^d}\int_{0}^{\infty}|n_m \log n_m|V(\theta)d\theta dx,
   \end{align*}
   using the $L^1 \cap L^\infty(\R^d)$ bound on $\rho_m$, for varying constants $C$ depending only on $\nu$, $V$, $\mu$, $r$, and the uniform in $m$ $L^1$ and $L^\infty$ bounds of $\rho_m.$ 
   Then, notice that $|n_m\log n_m|$ can be written as
   $$
   |n_m \log n_m|=n_m \log n_m +2|n_m \log n_m|_{\chi_{0\leq n_m \leq 1}}.
   $$
   Using the property $|s\log s|_{\chi_{0\leq s \leq 1}}\leq sC_1+C_2$ for arbitrary constants $C_1, C_2>0,$ we have 
   \begin{align*}
   \int_{\R^d}\int_{0}^{\infty}\partial_t (n_m \log n_m)V(\theta)d\theta dx &\leq C(T)+C\int_{\R^d}\int_{0}^{\infty}n_m \log n_m V(\theta)d\theta dx + C\int_{\R^d}\int_{0}^{\infty}n_m C_1 + C_2\,d\theta dx \\
   &\leq C(T)+C\int_{\R^d}\int_{0}^{\infty}n_m \log n_m V(\theta)d\theta dx +C||n_m||_{L^\infty(0,T;L^1([0,\infty)\times \R^d))}\\&+C|\supp n_m|\\
   &\leq C(T)+C\int_{\R^d}\int_{0}^{\infty}n_m \log n_m V(\theta)d\theta dx,
   \end{align*}
   using the compact support of $n_m$ in $\theta$ and $x$ and the $L^\infty(0,T;L^1( [0,\infty)\times \R^d)$ bound of $n_m$. By Gronwall's inequality,
   $$
   \int_{\R^d}\int_{0}^{\infty}n_m \log n_m d\theta dx \leq C(T)+C(T)\int_{\R^d}\int_{0}^{\infty}n_{0,m}\log n_{0,m}d\theta dx\leq C(T),
   $$
   using the uniform in $m$ bound of $n_{0,m}.$ Finally, using again the property $|s\log s|_{\chi_{0\leq s \leq 1}}\leq s\omega +Ce^{-\omega/2}$, we have
   \begin{align*}
       n_m \log_+ n_m& =n_m\log n_m+|n_m \log n_m|_{\chi_{0\leq n_m \leq 1}}\\
       &\leq n_m \log n_m +(n_m \omega +Ce^{-\omega/2}),
   \end{align*}
   and using again the compact support of $n_m$, and the $L^\infty(0,T;L^1([0,\infty)\times \R^d))$ bounds of $n_m$, 
   $$
   \int_{\R^d}\int_{0}^{\infty}n_m \log_+ n_m d\theta dx \leq C(T),
   $$
   for a constant $C(T)$ depending only on the uniform bound of the initial data.
 \end{proof}
\begin{proof}[Proof of Proposition \ref{prop.conv}]
    The equation \eqref{nm} can be written in the weak format
    \begin{multline}\label{weak.ndep}
    \int_{0}^{T}\int_{\R^d} \int_{0}^{\infty}n_m [-\partial_t \varphi - r(p_m)\partial_\theta \varphi + \nabla p_m \cdot \nabla \varphi +r(p_m)\nu(\theta,p_m)\varphi+\mu(\theta)\varphi]\, d\theta dx dt\\ = \int_{\R^d} \int_{0}^{\infty}n_{0,m}(x,\theta) \varphi(0,\theta,x) \, d\theta dx +\int_{0}^{T}\int_{\R^d}\int_{0}^{\infty}2\nu(\theta,p_m)n_m\varphi(t,0,x)\, d\theta dx dt,
\end{multline}  for any $\varphi \in C^1([0,T]\times [0,\infty)\times \R^d).$
Since $n_m$ converges weakly in $L^\infty(0,T;L^1((0,\infty)\times \R^d)$, $p_m$ and $\nabla p_m$ converge strongly in $L^2([0,T]\times \R^d),$ and solutions are compactly supported, all of the above terms can be passed to the limit, which solves
\begin{multline}\label{weak.ndep.inf}
    \int_{0}^{T}\int_{\R^d} \int_{0}^{\infty}n_\infty [-\partial_t \varphi - r(p_\infty)\partial_\theta \varphi + \nabla p_\infty \cdot \nabla \varphi +r(p_\infty)\nu(\theta,p_\infty)\varphi+\mu(\theta)\varphi]\, d\theta dx dt \\= \int_{\R^d} \int_{0}^{\infty}n_{0}(x,\theta) \varphi(0,\theta,x) \, d\theta dx +\int_{0}^{T}\int_{\R^d}\int_{0}^{\infty}2\nu(\theta,p_\infty)n_\infty\varphi(t,0,x)\, d\theta dx dt.
\end{multline}
Further, the equation on $\rho_m,p_m$ can be written in the weak format
\begin{multline*}
-\int_{0}^{T}\int_{\R^d}\rho_m \partial_t \varphi \,dx dt+\int_{0}^{T}\int_{\R^d}\nabla \varphi \cdot \nabla \rho_m^m\,dxdt-\int_0^T \int_{\R^d}\int_{0}^{\infty}n_m(F(\theta,p_m)-\mu(\theta)V(\theta))\,d\theta \varphi \, dx dt \\=\int_{\R^d}\rho_{0,m}(x)\varphi(0,x)\,dx
\end{multline*}
for $\varphi \in C^1([0,T]\times \R^d).$ Passing to the limit, 
\begin{multline*}
    -\int_{0}^{T}\int_{\R^d}\rho_\infty \partial_t \varphi \,dx dt+\int_{0}^{T}\int_{\R^d}\nabla \varphi \cdot \nabla p_\infty \,dxdt-\int_0^T \int_{\R^d}\int_{0}^{\infty}n_\infty(F(\theta,p_\infty)-\mu(\theta)V(\theta))\,d\theta \varphi \, dx dt \\=\int_{\R^d}\rho_{0}(x)\varphi(0,x)\,dx.
\end{multline*}
\end{proof}
Finally, we prove the complementarity formula, \eqref{comp}.
\begin{proof}[Proof of Theorem \ref{complement}]
This proof relies on the strong convergence of $p_m$ and $\nabla p_m$.
Dividing \eqref{p.eq} by $m-1$, for a test function $\xi \in C^1([0,T]\times R^d),$ the following identity holds.
\begin{align*}
\int_{0}^{T}\int_{\R^d}\xi p_m\Delta p_m &+\xi \frac{m}{m-1}(\frac{m-1}{m}p_m)^{\frac{m-2}{m-1}}\int_{0}^{\infty}n_m(F(\theta,p_m)-\mu(\theta)V(\theta))d\theta\,dx\,dt\\&=\frac{1}{m-1}\int_{0}^{T}\int_{\R^d} -p_m\partial_t \xi -\xi|\nabla p_m|^2\;dx\,dt.
\end{align*}
Taking the $\limsup$ of both sides, from the boundedness of $p_m \in L^1([0,T]\times \R^d)$ and $\nabla p_m \in L^2([0,T]\times \R^d)$,
\begin{align*}
\limsup_{m\rightarrow \infty}\int_{0}^{T}\int_{\R^d}\xi p_m\Delta p_m &+\xi \frac{m}{m-1}(\frac{m-1}{m}p_m)^{\frac{m-2}{m-1}}\int_{0}^{\infty}n_m(F(\theta,p_m)-\mu(\theta)V(\theta))d\theta\,dx\,dt=0.
\end{align*}
Integrating by parts and using the strong convergence of $p_m$ and $\nabla p_m$, and the weak convergence of $n_m$,
\begin{align*}
0&=\limsup_{m\rightarrow\infty}\int_{0}^{T}\int_{\R^d}-\nabla \xi \cdot \nabla p_m \,p_m-\xi |\nabla p_m|^2\\&+\xi(\frac{m}{m-1})^{\frac{1}{m-1}}p_m^{\frac{m-2}{m-1}}\int_{0}^{\infty}n_m(F(\theta,p_m)-\mu(\theta)V(\theta))d\theta\,dx\,dt\\
&=\int_{0}^{T}\int_{\R^d} -\nabla \xi \cdot \nabla p_\infty p_\infty - \xi |\nabla p_\infty|^2+\xi p_\infty \int_{0}^{\infty}n_\infty(F(\theta,p_\infty)-\mu(\theta)V(\theta))d\theta\, dx\, dt.
\end{align*}
 Integrating by parts again yields
\begin{align*}
     0=\int_{0}^{T}\int_{\R^d}\xi \Delta p_\infty p_\infty +\xi p_\infty \int_{0}^{\infty}n_\infty(F(\theta,p_\infty)-\mu(\theta)V(\theta))d\theta,
\end{align*}
 which completes the proof.
\end{proof}
\bibliographystyle{abbrv}
\bibliography{twref.bib}

\end{document}